 \documentclass[authoryear,preprint,review,10pt]{elsarticle}

\usepackage{amssymb}
\newtheorem{theorem}{Theorem}

\newtheorem{definition}[theorem]{Definition}

\newtheorem{lemma}[theorem]{Lemma}

\newenvironment{proof}[1][Proof]{\textbf{#1.} }{\ \rule{0.5em}{0.5em}}

\begin{document}

\begin{frontmatter}


% %\title{Title\tnoteref{label1}}

\title{Controllability of Neutral Stochastic Functional Integro-Differential Equations Driven by Fractional
  Brownian Motion with Hurst Parameter Lesser than 1/2}
\author{Brahim Boufoussi}
\ead{boufoussi@uca.ac.ma}
\author{Soufiane Mouchtabih\corref{1}}
\ead{soufiane.mouchtabih@gmail.com}

\address{LIBMA Laboratory, Department of Mathematics, Faculty of Sciences Semlalia, Cadi Ayyad University, 2390 Marrakesh, Morocco}
%
%\address[2]{Department of Mathematics, Regional Center for the Professions of Education and Training, Marrakesh, Morocco}
%\address[3]{ National School of Applied Sciences, Cadi Ayyad University, Safi, 46000, Morocco }
\cortext[1]{Corresponding author.}
\begin{abstract}
In this article we investigate the controllability for neutral stochastic functional integro-differential equations with finite delay, driven by a  fractional  Brownian motion with Hurst parameter lesser than  $1/2$  in a Hilbert space. We employ the theory of resolvent operators developed by \cite{Grimmer} combined with the Banach fixed point theorem to establish sufficient conditions to prove the desired result.
\end{abstract}


\end{frontmatter}



\section{Introduction}

The theory of controllability has been widely examined by many researchers due to various applications in the industry, biology and physics... It plays a vital role in both deterministic and stochastic control systems. In the literature, there are many different notions of controllability, both for linear and non-linear dynamical systems. Controllability of the deterministic and stochastic dynamical control systems in infnite-dimensional spaces is well-developed using different kind of approaches. It should be mentioned that there are few works in controllability problems for different kind of systems described by differential equations driven by fractional Brownian motion in Hilbert space with Hurst parameter $H\in (\frac{1}{2},1)$. For example, \cite{Ahmed} discussed the controllability of impulsive neutral functional SDEs, \cite{Lakhel} investigated the controllability result for neutral stochastic delay functional integro-differential equations, \cite{tam} studied the approximate controllability of a class of fractional stochastic differential equations driven by mixed fractional Brownian motion in Hilbert space. We would like to point out that there is no work reported yet on the controllability of neutral stochastic delay integro-differential equations perturbed by a fractional Brownian motion with Hurst parameter lesser than $\frac{1}{2}$.\\
After this brief outline on the literature, we will now describe precisely the system
investigated in this paper. Motivated by these works, we consider the following neutral stochastic functional integro-differential equation with finite delay:
\begin{eqnarray}\label{eq1}
\left\{
\begin{array}{r c l}
d[x(t)+g(t,x(t-r(t)))]=\big[A[x(t)+g(t,x(t-r(t)))]+Lu(t)\big]dt&&\\
+\left[\int_0^tB(t -s)\left[x(s)+g(s,x(s-r(s)))\right]ds+f(t,x(t-\rho(t)))\right]dt&& \\
+\sigma(t)dB^H(t),\qquad\qquad\qquad\;0\leq t \leq T,&&\\
x(t)=\varphi (t), \; -\tau \leq t \leq 0\,,
\end{array}
\right.
\end{eqnarray}
where $A:D(A)\subset X \rightarrow X$  is a closed linear operator, for all $t\geq 0,\, B(t)$ is a closed linear operator with domain $ D(B(t))\supset D(A)$. The control function $u(.)$ takes values in $L^2([0,T],U)$, the Hilbert space of admissible control functions for a separable Hilbert space $U$. $L$ is a bounded linear operator form $U$ into $X$. $B^H$ is a fractional Brownian motion with Hurst parameter $H<1/2$ on a real and separable Hilbert space $Y$. $r,\;  \rho:[0,+\infty)\rightarrow [0,\tau]\; (\tau  >0)$ are continuous and $f,g:[0,+\infty)\times X \rightarrow X,\;\; \sigma:[0,+\infty] \rightarrow \mathcal{L}_2^0(Y,X)$ are appropriate functions. Here $\mathcal{L}_2^0(Y,X)$  denotes the
space of all $Q$-Hilbert-Schmidt operators from $Y$ into $X$ (see section 2  below). We mention that a variant of this  equation without the term involving the operator $B(t)$ has been studied in  \cite{boufoussi} by using the  theory  of analytic semi-groups and  fractional powers  associated to its generator.

%%%%%%%%%%%%%%%%%%%%%%%%%%%%%%%%%%%%%%%%%%%%%%%%%%%%%%%%%%%%%%%%%%%%%%%%%%%%%%%%%%%%%%%%%%%
The outline of this paper is as follows. In Section 2 we introduce some notations, concepts, and basic results about fractional Brownian motion, Wiener integral over Hilbert spaces and we recall some preliminary results about resolvent operators. Section 3 investigates the controllability of the system $(\ref{eq1})$ by using Banach fixed point theorem. An illustrative example is given in the last Section.
%%%****************************************************************************
\section{Preliminaries}

In this section we collect some notions and conceptions on Wiener integrals with respect to an infinite dimensional fractional Brownian
 and we recall some basic results about resolvent operators which will be used throughout the whole of this paper.\\
%********************************************************************************************************************************
%****************************************************************************************
Let $(\Omega,\mathcal{F}, \mathbb{P})$ be a complete probability space.
Consider  a time interval $[0,T]$ with arbitrary fixed horizon $T$ and let $\{\beta^H(t) , t \in [0, T ]\}$
the one-dimensional  fractional Brownian motion with
Hurst parameter $H\in(0,1/2)$. This means
by definition that $\beta^H$ is a centred Gaussian process with covariance function:
$$ R_H(s, t) =\frac{1}{2}(t^{2H} + s^{2H}-|t-s|^{2H}).$$
 Moreover $\beta^H$ has the following Wiener
integral representation:
\begin{equation}\label{rep}
\beta^H(t) =\int_0^tK_H(t,s)d\beta(s)\,,
 \end{equation}
where $\beta = \{\beta(t) :\; t\in [0,T]\}$ is a Wiener process, and $K_H(t; s)$ is a square integrable kernel given by
(see \cite{nualart})
\begin{equation}\label{K}
K_H(t, s )=c_H\left[(\frac{t}{s})^{H-1/2}(t-s)^{H-1/2}-(H-\frac{1}{2})s^{1/2-H}\int_s^tu^{H-3/2} (u-s)^{H-1/2}du\right]
\end{equation}
for $t>s$,
where $c_H=\sqrt{\frac{2H}{(1-2H)\beta (1-2H,H+\frac{1}{2})}}$
and $\beta(,)$ is the Beta function.
We put $K_H(t, s ) =0$ if $t\leq s$.
And from $(\ref{K})$ it follows that:
\begin{equation}\label{K1}
| K(t,s)|\leq 2c_H\left((t-s)^{H-\frac{1}{2}}+
s^{H-\frac{1}{2}}\right)\,.
 \end{equation}
In the sequel we will use the following inequality :
\begin{equation}\label{K2}
|\frac{\partial K}{\partial t}(t,s)|\leq
c_H(\frac{1}{2}-H)(t-s)^{H-\frac{3}{2}}\,.
\end{equation}
We denote by $\mathcal{H}$ the
closure of set of indicator functions
$\{1_{[0;t]},  t\in[0,T]\}$ with respect to the scalar product
$\langle 1_{[0,t]},1_{[0,s]}\rangle _{\mathcal{H}}=R_H(t , s).$\\
The mapping $1_{[0,t]}\rightarrow \beta^H(t)$
 can be extended to an isometry between $\mathcal{H}$
and the first  Wiener chaos and
we will denote by $\beta^H(\varphi)$ the image of $\varphi$ by
the previous isometry.
%%%%%%%

It's known that $\mathcal{H}= I_{T^-}^{1/2-H}(L^2)$ and
$\mathcal{C}^\gamma([0,T])\subseteq \mathcal{H} $ if
$ \gamma > 1/2-H$
where $\mathcal{C}^\gamma([0,T])$ is the space of
$\gamma$-H\"{o}lder continuous functions and
$I_{T^-}^{\alpha}(L^2)$ is the image of $L^2([0,T])$ by the operator
$I_{T^-}^{\alpha}$ defined by:
 $$I_{T^-}^{\alpha}f(x)=\frac{1}{\Gamma (\alpha)}
\int_x^T (y-x)^{\alpha -1}f(y)dy\,.$$
%%%%%%%
Let us consider the operator $K_{H,T}^*$ from $\mathcal{H}$ to
$L^2([0,T])$ defined by
\begin{equation}\label{KH}
(K_{H,T}^*\varphi)(s)=K(T,s)\varphi(s)+\int_s^T(\varphi(r)-\varphi(s))
\frac{\partial K}{\partial r}(r,s)dr\,.
\end{equation}
We refer to \cite{nualart} for the proof of the fact that
$K_{H,T}^{*}$ is an isometry between $\mathcal{H}$ and $L^2([0,T])$.
Moreover for any $\varphi \in \mathcal{H}$, we have
$$\int_0^T {\varphi(s) d\beta^{H}(s)}:=\beta^H(\varphi)=
\int_0^T(K_{H,T}^*\varphi)(t)d\beta(t)\,.$$
We also have for  $ 0\leq t\leq T $
$$\int_0^t {\varphi(s) d\beta^{H}(s)}:=\int_0^T(K_{H,T}^*\, \varphi\, 1_{[0,t]})(s)d\beta(s)=\int_0^t(K_{H,t}^*\varphi)(s)d\beta(s)
\,,$$
where $K_{H,t}^*$ is defined in the same way as in (\ref{KH}) with
$t$ instead of $T$. In the next we will use the notation
$K_{H}^*$ without specifying the parameter $t\in[0,T]$.

%*********************************************************
Let $X$ and $Y$ be two real, separable Hilbert spaces and
let $\mathcal{L}(Y,X)$ be the space of bounded linear operator
from $Y$ to $X$. For the sake of convenience, we shall use the
same notation to denote the norms in $X,Y$ and $\mathcal{L}(Y,X)$.
Let $Q\in \mathcal{L}(Y,Y)$ be an operator defined by $Qe_n=\lambda_n e_n$
with finite trace $trQ=\sum_{n=1}^{\infty}\lambda_n<\infty$. where $\lambda_n \geq 0 \; (n=1,2...)$ are non-negative real numbers and $\{e_n\}\;(n=1,2...)$ is a complete orthonormal basis in $Y$.
 We define the infinite dimensional fBm on $Y$ with covariance
 $Q$ as
 $$B^H(t)=B^H_Q(t)=\sum_{n=1}^{\infty}\sqrt{\lambda_n}e_n\beta_n^H(t)\,,$$
 where $\beta_n^H$ are real, independent fBm's. This process is a $Y$-valuad Gaussian, it starts from $0$, has zero mean and covariance:
 $$E\langle B^H(t),x\rangle\langle B^H(s),y\rangle=R(s,t)\langle Q(x),y\rangle \;\; \mbox{for all}\; x,y \in Y \;\mbox {and}\;  t,s \in [0,T]$$
In order to define Wiener integrals with respect to the $Q$-fBm,
we introduce the space $\mathcal{L}_2^0:=\mathcal{L}_2^0(Y,X)$
of all $Q$-Hilbert-Schmidt operators $\psi:Y\rightarrow X$.
We recall that $\psi \in \mathcal{L}(Y,X)$ is called a $Q$-Hilbert-Schmidt operator, if
$$  \|\psi\|_{\mathcal{L}_2^0}^2:=\sum_{n=1}^{\infty}\|\sqrt{\lambda_n}\psi e_n\|^2 <\infty\, ,$$
and that the space $\mathcal{L}_2^0$ equipped with the inner
product
$\langle \varphi,\psi \rangle_{\mathcal{L}_2^0}=\sum_{n=1}^{\infty}\langle \varphi e_n,\psi e_n\rangle$ is a separable Hilbert space.

Now, let $\{\phi(s);\,s\in [0,T]\}$ be a function with values in $\mathcal{L}_2^0(Y,X)$, The Wiener integral of $\phi$ with respect to $B^H$ is defined by
\begin{equation}\label{int}
\int_0^t\phi(s)dB^H(s)=\sum_{n=1}^{\infty}\int_0^t \sqrt{\lambda_n}\phi(s)e_nd\beta^H_n(s)=\sum_{n=1}^{\infty}\int_0^t \sqrt{\lambda_n}(K_H^*(\phi e_n)(s)d\beta_n(s)
\end{equation}
where $\beta_n$ is the standard Brownian motion used to  present $\beta_n^H$ as in $(\ref{rep})$, and the above sum is finite when $\displaystyle\sum_n\lambda_n \| K_H^*(\phi e_n)\|^2 <\infty$  .\\
%*************************************************************************************************************************************
Now we turn to state some notations and basic facts about  the theory of  resolvent operators needed in the sequel.  For additional details on resolvent operators, we refer  to  \cite{Grimmer} and \cite{pruss}.\\

Let $A:D(A)\subset X \rightarrow X$   be a closed linear operator and  for all $t\geq 0,\, B(t)$  a closed linear operator with domain $ D(B(t))\supset D(A)$.
Let us denote by $Z$ the Banach space $D(A)$, the domain of operator $A$, equipped with the graph norm
$$\|y\|_Z :=\|Ay\|+\|y\| \;\;\mbox{for}\;\; y\in Z.$$
Let us  consider the following Cauchy problem
\begin{eqnarray}\label{cauchy}
\left\{
\begin{array}{r c l}
v'(t) &=& Av(t)+\int_0^tB(t -s)v(s)ds \;\; \mbox{for}\;\; t\geq 0,\\
v(0) &=& v_0 \in X.
\end{array}
\right.
\end{eqnarray}
\begin{definition}(\cite{Grimmer}) A resolvent operator of the Eq.$(\ref{cauchy})$ is a bounded linear operator valued function $R(t)\in \mathcal{L}(X)$ for $t\geq 0$, satisfying the following properties:
\begin{itemize}
\item  [(i)] $ R(0) = I$ and $\|R(t)\|\leq Ne^{\beta t}$  for some constants $N$ and $\beta$.
\item  [(ii)] For each $x\in X$, $R(t)x$ is strongly continuous for $t\geq 0$.
\item   [(iii)] For $x \in Z$, $R(.)x\in \mathcal{C}^1([0,+\infty);X)\cap \mathcal{C}([0,+\infty);Z)$ and
$$R'(t)x = AR(t)x +\int_0^tB(t -s)R(s)xds= R(t)Ax+\int_0^tR(t -s)B(s)xds, \;\;\mbox{for}\;\; t\geq 0.$$
\end{itemize}
\end{definition}
The resolvent operator plays an important role to study the existence of solutions and to establish a variation of constants formula for non-linear systems. For this reason, to assure the existence of the resolvent operator, we make the following hypotheses:
\begin{itemize}
\item [$(\mathcal{H}.1)$] $A$ is the infinitesimal generator of a $C_0$-semigroup $(T(t))_{t\geq0}$ on $X$.
 \item [$(\mathcal{H}.2)$] For all $t\geq 0$, $B(t)$ is a continuous linear operator from $(Z,\|.\|_Z)$ into $(X,\|.\|_X)$. Moreover, there is a locally integrable function  $c :\mathbb{R}^+\rightarrow  \mathbb{R}^+$ such that for any $y\in Z, t\mapsto B(t)y$ belongs to $W^{1,1}([0,+\infty), X)$ and	
     			$$\|\frac{d }{dt} B(t)y\|_X \leq c(t)\|y\|_Z \;\mbox{for}\;\; y \in Z\; \;\mbox{and}\;\; t \geq 0.$$
     \end{itemize}
     \begin{theorem}(\cite{desch}) Assume that hypotheses $(\mathcal{H}.1)$ and $(\mathcal{H}.2)$ hold.   Then  the Cauchy problem  $(\ref{cauchy})$  admits a unique resolvent operator $(R(t))_{t\geq 0}$.
     \end{theorem}
The following lemma proves that the resolvent operator $(R(t))_{t\geq 0}$ satisfies a Lipschitz condition:
\begin{lemma}\label{lem1}
Under conditions $(\mathcal{H}.1)$ and $(\mathcal{H}.2)$, we have:
$$ \|R(t)x-R(s)x\|\leq M \mid t-s\mid \| x\|_Z\;\; \mbox{for all } \;\;t,s \in [0,T]\;\;\mbox{and}\;\; x \in Z,$$
where  $M= \left(1+T\|B(0)\|+T\int_0^T c(s)ds \right)\sup_{t\in [0,T]}\|R(t)\|.$
\end{lemma}
\begin{proof}
Let $t,s \in [0,T]$ and $x \in Z$. By assumption $(\mathcal{H}.2)$, we have
\begin{eqnarray*}
\|B(t)x\|_X &\leq & \|B(0)x\|+\int_0^tc(u)du\|x\|_Z\nonumber\\
 &\leq & \left(\|B(0)\|+\int_0^T c(u)du\right)\|x\|_Z
\end{eqnarray*}
and
 \begin{eqnarray*}
\|R'(t)x\|_X &\leq & \sup_{u\in [0,T]}\|R(u)\| \|Ax\|+\sup_{u\in [0,T]}\|R(u)\|\int_0^T\|B(u)x\|_Xdu\nonumber\\
 &\leq & \sup_{u\in [0,T]}\|R(u)\|\left(1+ T\|B(0)\|+T\int_0^T c(u)du \right)\|x\|_Z
\end{eqnarray*}
which entails that
$\|R(t)x-R(s)x\|\leq M \mid t-s\mid \parallel x\parallel_Z$.
\end{proof}
%****************************************************************************************************************************
%===================================================================================================================================
\section{Main Result}
The following part of this paper moves on to prove the controllability of the stochastic system (\ref{eq1}).
For this task we assume that the following conditions are in force.
\begin{itemize}
\item [$(\mathcal{H}.3)$]The function $f:[0,+\infty)\times X
\rightarrow X$ satisfies the following Lipschitz
conditions: that is, there exist positive constants
$ C_i:=C_i(T), i=1,2$ such that, for all $t\in [0,T]$ and
$x,y\in X $
\item[] \qquad $ \|f(t,x)-f(t,y)\|\leq C_1 \|x-y\| \,\,,\,
\,\,\, \|f(t,x)\|^2\leq C_2 (1+\|x\|^2).$
\item [$(\mathcal{H}.4)$]
The function $g:[0,+\infty)\times X
\rightarrow X$ satisfies the following conditions:
\begin{itemize}
 \item [(i)]
There exist constants $ C_i:=C_i(T), i=3, 4$ such
 that  for all  $t\in [0,T]$ and $x,y\in  X$
\item[] \qquad $ \|g(t,x)-g(t,y)\|\leq C_3 \|x-y\| \,\,,\,
\,\,\, \|g(t,x)\|^2\leq C_4 (1+\|x\|^2).$
\item [(ii)]The function $g$ is continuous
in the quadratic mean sense:
$$\forall x\in \mathcal{C}([0,T], \mathbb{L}^2(\Omega, X)),\;\;\lim_{t\rightarrow s}\mathbb{E}\|g(t,x(t))-g(s,x(s))\|^2=0.$$
\end{itemize}
\item [$(\mathcal{H}.5)$]The function $\sigma:[0,+\infty)\rightarrow \mathcal{L}_2^0(Y,X)$ satisfies the following conditions:
 \begin{itemize}
 \item [(i)] There exists a constant $C_5>0$ such that, for all  $t,s \in [0,T]$
 $$\|\sigma(t)-\sigma (s)\|_{\mathcal{L}_2^0}\leq C_5 \mid t-s\mid^{\gamma}, \;\;\; \mbox{ where}\;\; \gamma>1/2-H.$$
 \item [(ii)]  $\forall  t \in [0,T]; \forall y\in D(A)$, $ \sigma(t)y\in D(A)$.
  \item [(iii)] There exists a constant $C_6>0$ such that $\int_0^T\|A\sigma(t)\|^2_{\mathcal{L}_2^0}\;dt \leq C_6$
 \end{itemize}
 \item [$(\mathcal{H}.6)$] The linear operator $W$ from $L^2([0,T],U)$ into $X$ defined by:
 $$Wu=\int_0^T R(T-s)Lu(s)ds$$
 has an inverse operator $W^{-1}$ that takes values in $L^2([0,T],U)\backslash KerW$, where $kerW=\{x\in L^2([0,T],U),\  Wx=0\}$, and there exists finite positive constants $M_L$ and $M_w$ such that $\|L\|\le M_L$ and $\|W^{-1}\|\le M_w$.
 \end{itemize}
Moreover, we assume that $\varphi \in \mathcal{C}([-\tau,0],\mathbb{L}^2(\Omega,X))$. 
%*****************************************************************
Similar to the deterministic situation we give the following
definition of mild solutions for equation (\ref{eq1}).
\begin{definition}
An $X$-valued  process $\{x(t),\;t\in[-\tau,T]\}$, is called a mild
solution of equation (\ref{eq1}) if
\begin{itemize}
\item[$i)$] $x(.)\in \mathcal{C}([-\tau,T],\mathbb{L}^2(\Omega,X))$,
\item[$ii)$] $x(t)=\varphi(t), \, -\tau \leq t \leq 0$.
\item[$iii)$]For arbitrary $t \in [0,T]$, we have
\begin{eqnarray*}
x(t)&=& R(t)(\varphi(0)+g(0,\varphi(-r(0))))-g(t,x(t-r(t)))\\
&+&\int_0^t R(t-s)[Lu(s)+f(s,x(s-\rho (s))]ds+\int_0^t R(t-s)\sigma(s)dB^H(s)\;\; \mathbb{P}-a.s.\phantom{\int_0^2+2}\\
\end{eqnarray*}
\end{itemize}
\end{definition}
The concept of controllability of neutral integro-differential stochastic functional differential equation is the following:
\begin{definition}
	The system $(\ref{eq1})$ is said to be controllable on the interval $[-\tau,T]$, if for every initial stochastic process $\varphi$ defined on $[-\tau,0]$ and $x_1\in X$, there exists a stochastic control $u\in L^2([0,T],U)$ such that the mild solution $x(.)$ of equation  $(\ref{eq1})$ satisfies $x(T)=x_1$. 
\end{definition}
%**********************************************************************************************************
%************************************************************************************************************
The main result of this work is given in the next theorem.
\begin{theorem}\label{th1}
Suppose that $(\mathcal{H}.1)-(\mathcal{H}.6)$ hold. Then, the system
$(\ref{eq1})$ is controllable on $[-\tau,T]$ provide that 
$$C_3^2+D^2C_1^2T^2+D^2M_L^2M_W^2C_3^2 T+D^4M_L^2M_W^2C_1^2T^3 < \frac{1}{4}.$$

\end{theorem}
%*****************************************************************
\begin{proof}
	Throughout the proof we will use the following notations:
	$$D:=\displaystyle\sup_{t\in [0,T]}\|R(t)\|\,,\,\,\widetilde{\sigma}:=\displaystyle\sup_{t\in [0,T]}\|\sigma(t)\|_{{\mathcal{L}_2^0}}.$$
Fix $T>0$ and let  $\mathcal{B}_T := \mathcal{C}([-\tau, T], \mathbb{L}^2(\Omega, X))$ be the Banach space of all continuous functions from $[-\tau, T]$ into $\mathbb{L}^2(\Omega, X)$,  equipped  with the  supremum norm $\|\xi\|_{\mathcal{B}_T}=\displaystyle\sup_{u \in [-\tau,T]}\left(\mathbb{E} \|\xi (u)\|^2\right)^{1/2}$ and let us consider the set
 $$S_T=\{x\in \mathcal{B}_T : x(s)=\varphi(s),\; \mbox {for} \;\;s \in [-\tau,0] \}.$$
 $S_T$ is a closed subset of $\mathcal{B}_T$ provided with the norm  $\|.\|_{\mathcal{B}_T}$.
 Thanks to hypothesis $(\mathcal{H}.6)$, we can define the following control: 
 \begin{eqnarray}
 u(t)&=& W^{-1}\{x_1-R(T)(\varphi(0) +g(0,\varphi(-r(0))))-g(T,\varphi(T-r(T))) \nonumber\\
 &&- \int_0^TR(T-s)f(s,x(s-\rho(s)))ds-\int_0^TR(T-s)\sigma(s)dB^H(s)\}(t).\quad 
 \end{eqnarray} 
 We define the operator $\psi$ on $S_T$ by:
$$\psi(x)(t)=\varphi(t)\,,\,\,\,\, \forall t\in [-\tau,0]\,, $$
and for all $t\in [0,T]$
 \begin{eqnarray*}
\psi(x)(t)&=&R(t)(\varphi(0)+g(0,\varphi(-r(0))))-g(t,x(t-r(t)))\\
&&+\int_0^t R(t-s)[Lu(s)+f(s,x(s-\rho (s)))]ds+\int_0^tR(t-s)\sigma(s)dB^H(s)\,.
 \end{eqnarray*}
Then, the controllability of system (\ref{eq1}) is equivalent to find a fixed point for the operator $\psi$. Next we will show by using Banach fixed point theorem that $\psi$ has a unique fixed point. We divide the subsequent proof into two steps.\\
{\bf Step 1.} For arbitrary $x\in S_T$, let us prove that $t\rightarrow \psi(x)(t)$ is continuous on the interval $[0, T]$ in the $\mathbb{L}^2(\Omega,X)$-sense.\\
Let us consider  $0 <t<T$  and $h>0$ small enough. Then for any fixed $x\in S_{T}$, we have
\begin{eqnarray*}
\mathbb{E}\|\psi(x)(t+h)-\psi(x)(t)\|^2&\le & 5\mathbb{E}\|(R(t+h)-R(t))[\varphi(0)+g(0,\varphi(-r(0)))]\|^2\\
&+& 5\mathbb{E}\|g(t+h,x(t+h-r(t+h)))-g(t+h,x(t-r(t)))\|^2\\
&+& 5\mathbb{E}\|\int_0^{t+h}R(t+h-s)f(s,x(s-r(s)))ds-\int_0^tR(t-s)f(s,x(s-r(s)))ds\|^2\\
&+& 5 \mathbb{E}\|\int_0^{t+h}R(t+h-s)\sigma(s)dB^H(s)-\int_0^tR(t-s)\sigma(s)dB^H(s)\|^2\\
&+& 5 \mathbb{E}\|\int_0^{t+h}R(t+h-s)Lu(s)ds-\int_0^tR(t-s)Lu(s)ds\|^2\\
&=& \sum_{1\le i\le 5}5 J_i(h).
\end{eqnarray*}

The continuity of the terms $J_1$, $J_2$ and $J_3$ can be proved by similar arguments as those used to prove Theorem $3.3$ in \cite{diop}. Then, it suffices to show that $J_4$ and $J_5$ possess the desired regularity. For the sake of clarity of the paper, we restrict us to the continuity of $J_4$. For the term $J_5$ thanks to the boundedness of the operators $L$ and $W^{-1}$, the same calculus provide the regularity.
\begin{eqnarray*}
J_4 &=& \mathbb{E}\|\int_0^{t+h}R(t+h-s)\sigma(s)dB^H(s)-\int_0^tR(t-s)\sigma(s)dB^H(s)\|^2 \\
&\leq & 2\mathbb{E}\|\int_0^t (R(t+h-s)-R(t-s))\sigma(s)dB^H(s)\|^2+2\mathbb{E}\|\int_t^{t+h} R(t+h-s)\sigma(s)dB^H(s)\|^2\\
&\leq & J_{41}(h)+J_{42}(h).
\end{eqnarray*}
By $(\ref{int})$, we get that
\begin{eqnarray}\label{1j}
J_{41}(h)&= &2\sum_{n=1}^{\infty}\lambda_n\int_0^t \|K_t^*(R(t+h-s)-R(t-s))\sigma(s)e_n\|^2ds\nonumber\\
&\leq & 4 \sum_{n=1}^{\infty}\lambda_n\int_0^t K^2(t,s)\|(R(t+h-s)-R(t-s))\sigma(s)e_n\|^2ds\nonumber\\
&+& 8\sum_{n=1}^{\infty}\lambda_n\int_0^t\left\|\int_s^t\left(R(t+h-r)-R(t+h-s)+R(t-s)-R(t-r)\right)\sigma(r)e_n\frac{\partial K}{\partial r}(r,s)dr\right\|^2ds \nonumber\\
 &+&8\sum_{n=1}^{\infty}\lambda_n\int_0^t\left\|\int_s^t (R(t+h-s)-R(t-s))(\sigma(s)e_n -\sigma(r)e_n)\frac{\partial K}{\partial r}(r,s)dr\right\|^2ds\nonumber\\
&\leq & I_{1}+I_{2}+I_{3}.
\end{eqnarray}
We estimate the various terms of the right-hand side of $(\ref{1j})$ separately.
For the first term, we have: $I_{1}= \sum_{n=1}^{\infty}f_n(h)$ where
 $$f_n(h)=4 \lambda_n\int_0^t  K^2(t,s)\|(R(t+h-s)-R(t-s))\sigma(s)e_n\|^2ds.$$
By using the strong continuity of $R(t)x$, we get:
 $$\lim_{h\rightarrow 0}  K^2(t,s)\|(R(t+h-s)-R(t-s))\sigma(s)e_n\|^2=0,$$
 and since
 \begin{eqnarray*}
 &&\lambda_n K^2(t,s)\|(R(t+h-s)-R(t-s))\sigma(s)e_n\|^2\\
 &&\qquad\qquad \qquad\qquad  \leq  4 D^2 \widetilde{\sigma}^2 K^2(t,s)\in \mathbb{L}^1((0,t),ds),
 \end{eqnarray*}
 then, we conclude by the Lebesgue dominated theorem that $\lim_{h\rightarrow 0} f_n(h)=0.$
 Besides, we have:
  $$ |f_n(h) |\leq 16D^2 \lambda_n\int_0^t  K^2(t,s)\|\sigma(s)e_n\|^2ds, $$
and since $$\sum_{n=1}^{\infty}16D^2 \lambda_n\int_0^t  K^2(t,s)\|\sigma(s)e_n\|^2ds\leq16D^2\widetilde{\sigma}^2\int_0^t  K^2(t,s)ds <\infty.$$
Then, we conclude by the double limit theorem that
 \begin{equation}\label{j1}
 \lim_{h\rightarrow 0} I_1= \lim_{h\rightarrow 0} \sum_{n=1}^{\infty}f_n(h)=\sum_{n=1}^{\infty} \lim_{h\rightarrow 0}f_n(h)=0.
 \end{equation}
For the second term, we have: $I_2= \sum_{n=1}^{\infty}g_n(h)$ where $$g_n(h)=8\lambda_n\int_0^t\left(\int_s^t\|\left(R(t+h-r)-R(t+h-s)+R(t-s)-R(t-r)\right)\sigma(r)e_n\|\frac{\partial K}{\partial r}(r,s)dr\right)^2ds.$$
The strong continuity of $R(t)x$ provides:
 $$\lim_{h\rightarrow 0}\|\left(R(t+h-r)-R(t+h-s)+R(t-s)-R(t-r)\right)\sigma(r)e_n\|\frac{\partial K}{\partial r}(r,s)=0.$$
Using Lemma $\ref{lem1}$ together with inequality $(\ref{K2})$, we get
\begin{eqnarray*}
 &&\|\left(R(t+h-r)-R(t+h-s)+R(t-s)-R(t-r)\right)\sigma(r)e_n\||\frac{\partial K}{\partial r}(r,s)|\\
&&\qquad\qquad \qquad\qquad \leq 2M C_H(1/2 -H)\|\sigma(r)e_n\|_{Z}(r-s)^{H-1/2}\in \mathbb{L}^1((s,t),dr)
\end{eqnarray*}
then, we conclude anew by the dominated convergence theorem that
 $$\lim_{h\rightarrow 0}\int_s^t\|\left(R(t+h-r)-R(t+h-s)+R(t-s)-R(t-r)\right)\sigma(r)e_n\|\frac{\partial K}{\partial r}(r,s)dr=0.$$
 Furthermore, Lemma $\ref{lem1}$ and inequality $(\ref{K2})$ entail 
 \begin{eqnarray*}
 &&\left(\int_s^t\|\left(R(t+h-r)-R(t+h-s)+R(t-s)-R(t-r)\right)\sigma(r)e_n\|\frac{\partial K}{\partial r}(r,s)dr\right)^2\\
 && \qquad\qquad \qquad\qquad \leq \frac{2M^2C_H^2(1/2-H)^2}{H}(t-s)^{2H}\int_0^t\|\sigma(r)e_n\|^2_{Z}dr \, \in \mathbb{L}^1((0,t), ds).
 \end{eqnarray*}
 Then we conclude by the Lebesgue dominated theorem that $ \lim_{h\rightarrow 0}g_n(h)=0.$\\
On account of: $$g_n(h)\leq  \frac{16M^2C_H^2(1/2-H)^2}{H(2H+1)}t^{2H+1}\int_0^t\lambda_n\|\sigma(r)e_n\|^2_{Z}dr,$$
and
$$\sum_{n=1}^{\infty}\int_0^t\lambda_n \|\sigma(r)e_n\|^2_{D(A)}dr\leq 2 T\widetilde{\sigma}^2 + 2 \int_0^t\|A\sigma(r)\|^2_{\mathcal{L}_2^0}\,dr\, < \infty,$$
we conclude by the double limit theorem that
\begin{equation}\label{j2}
\lim_{h\rightarrow 0} I_{2}= \lim_{h\rightarrow 0} \sum_{n=1}^{\infty}g_n(h)=\sum_{n=1}^{\infty} \lim_{h\rightarrow 0}g_n(h)=0.
\end{equation}
Similar computations can be used to estimate the term $ I_{3}$, indeed, we have: $ I_{3}=\sum_{n=1}^{\infty}l_n(h)$, where
$$l_n(h)=8\lambda_n\int_0^t\left(\int_s^t \|(R(t+h-s)-R(t-s))(\sigma(s)e_n -\sigma(r)e_n)\|\frac{\partial K}{\partial r}(r,s)dr\right)^2ds.$$
Again, the strong continuity of $R(t)x$ gives us:
 $$\lim_{h\rightarrow 0} \|(R(t+h-s)-R(t-s))(\sigma(s)e_n -\sigma(r)e_n)\|\frac{\partial K}{\partial r}(r,s)=0.$$
By assumption $(\mathcal{H}.5)$ and inequality $(\ref{K2})$, we have
 $$ \|(R(t+h-s)-R(t-s))(\sigma(s)e_n -\sigma(r)e_n)\|\frac{\partial K}{\partial r}(r,s)\leq \frac{2DC_5C_H(1/2-H)}{\sqrt{\lambda_n}}(r-s)^{\gamma+ H-3/2}\in \mathbb{L}^1((s,t),dr)$$
  Once more, we conclude by the Lebesgue dominated theorem that:
  $$\lim_{h\rightarrow 0}\int_s^t \|(R(t+h-s)-R(t-s))\sigma(s)e_n -\sigma(r)e_n\|\frac{\partial K}{\partial r}(r,s)dr=0.$$
  On the other hand, we have
   \begin{eqnarray*}
  &&\left(\int_s^t \|(R(t+h-s)-R(t-s))(\sigma(s)e_n -\sigma(r)e_n)\|\frac{\partial K}{\partial r}(r,s)dr\right)^2\\
  &&\qquad\qquad \qquad\qquad \leq \frac{4D^2C_5^2C_H^2(1/2-H)^2}{\lambda_n(\gamma+H-1/2)^2}(t-s)^{2\gamma+2H-1}\in \mathbb{L}^1((0,t),ds).
 \end{eqnarray*}
 One more time, the Lebesgue dominated theorem gives:
 \begin{equation}\label{z1}
 \lim_{h\rightarrow 0}l_n(h)=0.
\end{equation}
In view of $(\ref{K2})$ we have 
\begin{equation}\label{z2}
l_n(h)\leq 32\lambda_nD^2C_H^2(1/2-H)^2\int_0^t\left(\int_s^t \|\sigma(s)e_n -\sigma(r)e_n\|(r-s)^{H-3/2}dr\right)^2ds.
\end{equation}
Now, let $\alpha \in (1, \gamma+H+1/2)$. By H\"{o}lder's inequality and  assumption $(\mathcal{H}.5)$, we get
\begin{eqnarray}\label{z3}
&&\sum_{n=1}^{\infty}\lambda_n\int_0^t\left(\int_s^t \|\sigma(s)e_n -\sigma(r)e_n\|(r-s)^{H-3/2}dr\right)^2ds\nonumber\\
&&\qquad\qquad \leq \int_0^t\left(\int_s^t(t-s)^{-3+2\alpha}dr\int_s^t \|\sigma(s) -\sigma(r)\|_{\mathcal{L}_2^0}^2(r-s)^{2H-2\alpha}dr\right)ds\nonumber\\
&&\qquad\qquad \leq    \frac{C_5^2}{(2\alpha-2)(2H+2\gamma-2\alpha+1)}\int_0^t(t-s)^{2\gamma+2H-1}\,ds<\infty
\end{eqnarray}
Combining inequalities $(\ref{z1})$, $(\ref{z2})$, $(\ref{z3})$ and the double limit theorem, we get that
\begin{equation}\label{j3}
\lim_{h\rightarrow 0} I_{3}= \lim_{h\rightarrow 0} \sum_{n=1}^{\infty}l_n(h)=\sum_{n=1}^{\infty} \lim_{h\rightarrow 0}l_n(h)=0.
\end{equation}
Inequalities $(\ref{j1})$, $(\ref{j2})$ and  $(\ref{j3})$ imply that $\displaystyle\lim_{h\rightarrow 0}J_{41}(h)=0$.
%%%%%%%%%%%%%%%%%%%%%%%%%%%%%%%%%%%%%%%%%%%%%%%%%%%%%%%%%%%%%%%%%%%%%%%%%%%%%%%%%%%%%%%%%%%%%%%%%%%%%%%%%%%%%%%%%

By the same token, we have
\begin{eqnarray*}
J_{42}(h)&= &2 \sum_{n=1}^{\infty}\lambda_n\int_t^{t+h} \|K_{t+h}^*(R(t+h-s)\sigma(s)e_n\|^2ds\\
&\leq & 4 \sum_{n=1}^{\infty}\lambda_n\int_t^{t+h} K^2(t+h,s)\|R(t+h-s)\sigma(s)e_n\|^2ds\\
&&+ 8\sum_{n=1}^{\infty}\lambda_n\int_t^{t+h}\left\|\int_s^{t+h}(R(t+h-r)-R(t+h-s))\sigma(r)e_n\frac{\partial K}{\partial r}(r,s)dr\right\|^2ds \\ &&+8\sum_{n=1}^{\infty}\lambda_n\int_t^{t+h}\left\|\int_s^t R(t+h-s)(\sigma(r)e_n -\sigma(s)e_n)\frac{\partial K}{\partial r}(r,s)dr\right\|^2ds.\\
&\leq & I'_{1}+I'_{2}+I'_{3}\,.
\end{eqnarray*}
By means of $(\ref{K1})$, we get
\begin{eqnarray}\label{l1}
 I'_{1}&\leq & 16D^2c_H^2 \sum_{n=1}^{\infty}\lambda_n\int_t^{t+h}\left((t+h-s)^{2H-1}+s^{2H-1}\right)\|\sigma(s)e_n \|^2ds\nonumber\\
&\leq &  \frac{8c_H^2D^2\widetilde{\sigma}^2}{H}
\left(h^{2H}+(t+h)^{2H}-t^{2H}\right).
\end{eqnarray}
Using H\"{o}lder's inequality, Lemma $\ref{lem1}$ together with inequality $(\ref{K2})$, we get
\begin{eqnarray}\label{l2}
& I'_{2}&\nonumber\\
&\leq&  8M^2c_H^2(1/2-H)^2 \sum_{n=1}^{\infty}\lambda_n
\int_t^{t+h}\left(\int_s^{t+h} (r-s)^{H-1/2}\|\sigma(r)e_n\|_{Z}dr\right)^2ds\nonumber\\
&\leq &  8M^2 c_H^2(1/2-H)^2 \sum_{n=1}^{\infty}\lambda_n\int_t^{t+h}\left(\int_s^{t+h}(r-s)^{2H-1}dr \int_s^{t+h} \|\sigma(r)e_n\|^2_{Z}dr\right)ds\nonumber\\
&\leq &  16M^2c_H^2(1/2-H)^2 \int_t^{t+h}\left(\frac{1}{2H}(t+h-s)^{2H} \int_s^{t+h} \left(\|A\sigma(r)\|_{\mathcal{L}_2^0}^2+\|\sigma(r)\|_{\mathcal{L}_2^0}^2\right)dr\right)ds\nonumber\\
&\leq & \frac{8M^2 c_H^2(1/2-H)^2h^{2H+1}}{H(2H+1)}\int_0^{T} \left(\|A\sigma(r)\|_{\mathcal{L}_2^0}^2+\|\sigma(r)\|_{\mathcal{L}_2^0}^2\right)dr.
\end{eqnarray}
Inequality $(\ref{K2})$, condition  $(\mathcal{H}.5)$ and H\"{o}lder's  inequality give
\begin{eqnarray}\label{l3}
& I'_{3}&\nonumber\\
&\leq & \delta \sum_{n=1}^{\infty}\lambda_n\int_t^{t+h}\left(\int_s^{t+h} \|\sigma(r)e_n-\sigma(s)e_n\|(r-s)^{H-3/2}dr\right)^2ds\nonumber\\
&\leq & \delta \,C_5^2 \sum_{n=1}^{\infty}\lambda_n\int_t^{t+h}
\left(\int_s^{t+h} (r-s)^{H-3/2+\gamma}dr \int_s^{t+h}
\|\sigma(r)e_n-\sigma(s)e_n\|^2(r-s)^{H-3/2-\gamma}dr\right)ds\nonumber\\
&\leq & \delta \,C_5^2 \int_t^{t+h}\left(\int_s^{t+h} (r-s)^{H-3/2+\gamma}dr\right)^2ds\nonumber\\
&\leq & \frac{\delta\, C_5^2}{2(H+\gamma)(H+\gamma-1/2)^2}h^{2(H+\gamma)}.
\end{eqnarray}
where  $\delta=8D^2 c_H^2(1/2-H)^2$.

 Inequalities $(\ref{l1})$, $(\ref{l2})$ and $(\ref{l3})$ imply that $\displaystyle\lim_{h\rightarrow 0}J_{42}(h)=0$.
Thus, we conclude that  the function  $t \rightarrow \psi(x)(t)$
is continuous on $[0,T]$ in the $\mathbb{L}^2$-sense.

%%%*********************************************************************************%
{\bf Step 2.}
Now, we are going to show that $\psi$ is a contraction mapping in $S_{T}$. Let $x,y\in S_T$, we obtain for any fixed $t\in [0,T]$
\begin{eqnarray*}
\|\psi(x)(t)-\psi(y)(t)\|^2&\leq& 4\|g(t,x(t-r(t)))-g(t,y(t-r(t)))\|^2\\
&+&4\|\int_0^tR(t-s)(f(s,x(s-\rho(s)))-f(s,y(s-\rho(s)))ds\|^2\\
&+&4\|\int_0^tR(t-v)LW^{-1}\left\lbrace g(T,x(T-r(T)))-g(T,y(T-r(T)))\right\rbrace (v)dv\|^2\\
&+& 4\|\int_0^tR(t-v)LW^{-1}\left\lbrace \int_0^T R(T-s)[f(s,x(s-\rho(s)))-f(s,y(s-\rho(s)))]ds\right\rbrace (v)dv\|^2
\end{eqnarray*}
By virtue of the boundedness of the operators $L$ and $W$, and Lipschitz property of $g$ and $f$ combined with H\"older's inequality, we obtain for all $t\in [0,T]$:
\begin{eqnarray*}
\mathbb{E}\|\psi(x)(t)-\psi(y)(t)\|^2 &\leq & 4C_3^2\mathbb{E}\|x(t-r(t))-y(t-r(t))\|^2\\
&+& 4tD^2C_1^2\int_0^t\mathbb{E}\|x(s-\rho(s))-y(s-\rho(s))\|^2ds\\
&+& 4tD^2M_L^2M_W^2C_3^2 \mathbb{E}\|x(T-r(T))-y(T-r(T))\|^2\\
&+& 4tD^4M_L^2M_W^2C_1^2T \int_0^T\mathbb{E}\|x(s-\rho(s))-y(s-\rho(s))\|^2 ds.
\end{eqnarray*}
Consequently, $$\sup_{s\in[-\tau,T]}\mathbb{E}\|\psi(x)(t)-\psi(y)(t)\|^2\leq
K\sup_{s\in[-\tau,T]} \mathbb{E}\|x(s)-y(s)\|^2,$$ where
$$ K=4[C_3^2+D^2C_1^2T^2+D^2M_L^2M_W^2C_3^2 T+D^4M_L^2M_W^2C_1^2T^3].$$
Hence $\psi$ is a contraction mapping on $S_{T}$ and therefore has a unique fixed
point, which is a mild solution of equation $(\ref{eq1})$ on $[-\tau,T]$. Clearly, $\psi(x)(T)=x_1$ which implies that the system $(\ref{eq1})$ is controllable on $[-\tau,T]$.
This completes the proof.
\end{proof}
\section{Example}
By way of illustration, we consider the following stochastic integro-differential equation with finite delays $\tau_1$ and $\tau_2$, $0\le \tau_1,\tau_2<\infty$, of the form:
\begin{eqnarray}\label{exp}
	\left\{
	\begin{array}{r c l}
		\frac{\partial}{\partial t}[x(t,\xi)+\hat{g}(t,x(t-\tau_1,\xi))] =\frac{\partial^2}{\partial^2\xi}[x(t,\xi)+\hat{g}(t,x(t-\tau_1,\xi))]&& \\
		+\int_0^tb(t-s)\frac{\partial^2}{\partial^2\xi}[x(s,\xi)+\hat{g}(s,x(s-\tau_1,\xi))]ds&& \\
		+\hat{f}(t,x(t-\tau_2,\xi))+\mu(t,\xi)+\sigma(t)\frac{dB^H}{dt}(t),\ t\ge 0&&\\
		x(t,0)+g(t,x(t-\tau_1,0))=0,\ t\ge 0,&&\\
		x(t,\pi)+g(t,x(t-\tau_1,\pi))=0,\ t\ge 0,&&\\
		x(s,\xi)=\varphi(s,\xi),\ -\tau\le s\le 0\  a.s.
	\end{array}
	\right.
\end{eqnarray}
where $B^H$ denotes a fractional Brownian motion, $\hat{f}$, $\hat{g}: \mathbb{R}_+\times\mathbb{R}\to \mathbb{R}$ are continuous functions and $b:\mathbb{R}_+\to \mathbb{R}$ is continuous function and $\varphi:[-\tau,0]\times[0,\pi]\to \mathbb{R}$ is a given continuous function such that $\varphi(s,.)\in L^2([0,\pi])$ is measurable and satisfies $\mathbb{E}\|\varphi\|^2<\infty$.\\
Let $X=Y=L^2([0,\pi])$. Define the operator $A:D(A)\subset X\to X$ given by $A=\frac{\partial^2}{\partial^2\xi}$ with domain:
$$D(A)=\{x\in X: x"\in X,\ x(0)=x(\pi)=0\},$$
Then,   
$$Ax=\sum_{n=1}^{\infty}n^2<x,e_n>_X e_n,\quad  x\in D(A)$$
where $e_n:=\sqrt{\frac{2}{\pi}}\sin nx$, $n=1,2,...$ is an orthogonal set of eigenvector of $-A$.\\
It is known that $A$ is the infinitesimal generator of a strongly continuous semigroup of bounded linear operators $(T(t))_{t\ge 0}$ on $X$, which is given by
$$T(t)x=\sum_{n=1}^{\infty}n^2<x,e_n>e_n.$$
Furthermore, $\|T(t)\|\le e^{-\pi^2t}$ for every $t\ge 0$. \\
%In order to define the operator $Q:Y:=L^2([0,\pi],\mathbb{R})\longrightarrow Y$, we choose a sequence of positive number $\{\lambda_n\}_{n\in \mathbb{N}}$, set $Qe_n=\lambda_n e_n$, and assume that $$tr(Q)=\sum_{n=1}^{\infty}\sqrt{\lambda_n}<\infty.$$
%Define the fractional Brownian in $Y$ by
%$$B^H(t)=\sum_{n=1}^{\infty}\lambda_n\beta^H(t)e_n$$
%where $H\in (0,\frac{1}{2})$ and $(\beta^H_t)$ is a sequence of one dimensional fractional Brownian motion mutually independent. \\
Let $B:D(A)\subset X\to X$ be the operator given by $B(t)x=b(t)Ax$, for $t\ge 0$ and $x\in D(A)$. 
Define the operator $W:L^2([0,T],U)\to X$ by:
$$Wu(\xi)=\int_0^TR(T-s)\mu(t,\xi)ds,\quad 0\le \xi\le \pi,$$
$W$ is a bounded linear operator but not necessarily one-to-one. Let
$$KerW=\{x\in L^2([0,T],U),\ Wx=0\}$$ 
be the null space of $W$ and $[KerW]^\perp$ be its orthogonal complement in $L^2([0,T],U)$. Let $\tilde{W}:[KerW]^\perp\to Range(W)$ be the restriction of $W$ to $[KerW]^\perp$, $\tilde{W}$ is one-to-one operator. The inverse mapping theorem says that $\tilde{W}^{-1}$ is bounded since $[KerW]^\perp$ and $Range(W)$ are Banach spaces. So that $W^{-1}$ is bounded and takes values in $L^2([0,T],U) \backslash  KerW$, hence assumption $(\mathcal{H}.6)$ is satisfied.
We suppose that:
\begin{itemize}
	\item[(i)] The operator $Lu:[0,T]\to X$, defined by: 
	$$Lu(t)(\xi)=\mu(t,\xi),\quad \xi\in [0,\pi],\quad u\in L^2([0,T],U).$$
	\item[(ii)]For $t\in [0,T]$, $\hat{f}(t,0)=\hat{g}(t,0)=0$,
	\item[(iii)]There exist positive constants $C_1$ and $C_3$, such that
	$$|\hat{f}(t,\xi_1)-\hat{f}(t,\xi_2)|\le C_1|\xi_1-\xi_2|,\ for t\in [0,T]\ and\ \xi_1, \xi_2 \in \mathbb{R},$$
	$$|\hat{g}(t,\xi_1)-\hat{g}(t,\xi_2)|\le C_1|\xi_1-\xi_2|,\ for t\in [0,T]\ and\ \xi_1, \xi_2 \in \mathbb{R}.$$ 
	\item[(iv)] There exist positive constants $C_2$ and $C_4$, such that
	$$|\hat{f}(t,\xi)|\le C_2(1+|\xi|^2),\ for\ t\in [0,T]\ and\ \xi \in \mathbb{R},$$
	$$|\hat{g}(t,\xi)|\le C_4(1+|\xi|^2),\ for\ t\in [0,T]\ and\ \xi \in \mathbb{R}.$$
	\item[(v)]The function $\sigma:[0,+\infty)\to \mathcal{L}^2_0\big(L^2([0,\pi],L^2([0,\pi]))\big)$ satisfies assumptions $(\mathcal{H}.6)$. 
\end{itemize}   
Define the operators $f,g:\mathbb{R}_+\times L^2([0,\pi])\to L^2([0,\pi])$ by
$$f(t,\phi)(\xi)=\hat{f}(t,\phi(-\tau_1)(\xi))\ for\ \xi \in [0,\pi]\ and\  \phi \in L^2([0,\pi]),$$
and
$$g(t,\phi)(\xi)=\hat{g}(t,\phi(-\tau_1)(\xi))\ for\ \xi \in [0,\pi]\ and\  \phi \in L^2([0,\pi]).$$
If we put:
\[
\left\{
\begin{array}{r c l}
x(t)(\xi)&=& x(t,\xi), \ for\ t\in [0,T]\ and\ \xi \in [0,\pi],\\
x(t)(\xi)&=& \varphi(t,\xi), , \ for\ t\in [-\tau,0]\ and\ \xi \in [0,\pi].
\end{array}
\right.
\]
Then, equation $(\ref{exp})$ takes the following abstract form:
\begin{eqnarray*}
	\left\{
	\begin{array}{r c l}
		&& d[x(t)+g(t,x(t-r(t)))]=\big[A[x(t)+g(t,x(t-r(t)))]+Lu(t)\big]dt\nonumber\\
		&&\qquad\qquad +\left[\int_0^tB(t -s)\left[x(s)+g(s,x(s-r(s)))\right]ds+ f(t,x(t-\rho(t)))\right]dt\nonumber\\
		&&\qquad \qquad+\sigma(t)dB^H(t),\qquad\qquad\qquad\qquad\qquad\qquad\;0\leq t \leq T,\nonumber\\
		&& x(t)=\varphi (t), \; -\tau \leq t \leq 0\,,
	\end{array}
	\right.
\end{eqnarray*}
Moreover, if $b$ is bounded and $\mathcal{C}^1$ such that $b'$ is bounded and uniformly continuous, then $(\mathcal{H}.2)$ is satisfied, hence equation $(\ref{exp})$ has a resolvent operator $(R(t))_{t\ge0}$ on $X$. Besides, the continuity of $\hat{f}$ and $\hat{g}$ and assumption $(ii)$ it ensues that $f$ and $g$ are continuous. In accordance with assumption $(iv)$ we obtain
$$\|f(t,\phi_1)-f(t,\phi_2)\|_{L^2([0,\pi])}\le C_1\|\phi_1-\phi_2\|_{L^2([0,\pi])},$$ 
$$\|g(t,\phi_1)-g(t,\phi_2)\|_{L^2([0,\pi])}\le C_3\|\phi_1-\phi_2\|_{L^2([0,\pi])}$$
Furthermore, by assumption $(iv)$, it follows that 
$$\|f(t,\phi)\|_{L^2([0,\pi])}\le C_2(1+\|\phi\|^2), and \quad \|f(t,\phi)\|_{L^2([0,\pi])}\le C_4(1+\|\phi\|^2).$$
Moreover, it is possible to choose the constants in such way  that:
$$4[C_3^2+D^2C_1^2T^2+D^2M_L^2M_W^2C_3^2 T+D^4M_L^2M_W^2C_1^2T^3]<1.$$

Thus, all the assumptions of Theorem $(\ref{th1})$ are fulfilled. Consequently, the system $(\ref{exp})$ is controllable on $[-\tau,T]$.
%%****************************************************************************


\section{References}
\begin{thebibliography}{30}

 \bibitem[Ahmed.(2015)]{Ahmed}
 {Ahmed, H.M}2015. {Controllability of impulsive neutral stochastic differential equations with fractional Brownian motion}, IMA J. Math. Control and Inf. 32(4), pp. 781-794.

\bibitem[Boufoussi and Hajji.(2017)] {boufoussi}
{Boufoussi, B. and Hajji, S.,} 2017. {Transportation
inequalities for neutral stochastic differential equations driven by   fractional  Brownian motion with Hurst parameter lesser than 1/2}. Mediterr. J. Math.  14:192 DOI 10.1007/s00009-017-0992-9 

%\bibitem[Boufoussi and Hajji.(2012)]{Boufoussi2}
%{Boufoussi, B. and Hajji, S.} 2012.{ Neutral stochastic functional differential equations driven by a fractional Brownian motion in a Hilbert space}, Statist Proba Lett. 82, pp. 1549-1558

%\bibitem[Caraballo \textit{et al}.(2011)]{carab}
%{Caraballo, T.,  Diop, M. A. and Taniguchi, T.} 2011. {The existence and exponential behaviour of solutions to stochastic delay evolution equations with a fractional Brownian motion}. Nonlinear Anal.:TMA 74, pp.3671-3684

\bibitem[Caraballo and Diop.(2013)] {diop}
{Caraballo, T.,  Diop, M. A., } 2013. { Neutral stochastic delay partial functional integro-differential equations driven by a fractional Brownian motion}.
 Front. Math. China, 8(4): 745-760

%\bibitem[Caraballo \textit{et al}.(2014)]{carab2}
%{Caraballo, T., Diop, M. A. and Ndiaye, A. A.}2014. {Asymptotic behavior of neutral stochastic partial functional integro-differential equations driven by a  fractional Brownian motion}, J. Nonlinear  Sci. Appl. 7, pp. 407-421

\bibitem[Desch \textit{et al}.(1984)] {desch}
{Desch, W., Grimmer, R.,  Schappacher, W.,}  1984. {Some consideration for linear integrodifferential equations.} J. Math. Anal. Appl. 104:219–234.

%\bibitem[Ducan et \textit{et al}.(2002)]{ducan}
%{Duncan, T. E., Maslowski, B. and Pasik-Duncan, B.}2002.{Fractional Brownian motion and stochastic equations in Hilbert spaces}, Stoch. Dyn. 2, pp. 225-250

\bibitem[Grimmer.(1982)]{Grimmer}
{Grimmer, R. C.,} 1982. { Resolvent opeators for integral equations in a Banach space.}  Transactions of the American Mathematical Society, 273:333-349.

\bibitem[Lakhel.(2016)]{Lakhel}
{Lakhel, E H.} 2016. {Controllability of neutral stochastic functional integro-differential equations driven by fractional Brownian motion}, Stoch. Ana. and App. 34(3), pp. 427-440.

%\bibitem[Maslowski and Nualart.(2003)]{MalNua}
%{ Maslowski B. and D. Nualart}2003. {Evolution equations driven by a fractional Brownian motion}. J. Funct. Anal 202, pp. 277-305

\bibitem[Nualart.(2006)] {nualart}
{Nualart, D.,} 2006.  { The Malliavin Calculus and Related Topics, second edition,} Springer-Verlag, Berlin.

\bibitem[Tamilalagan and Balasubramanniam.(2017)]{tam}
{Tamilalagan, P. and Balasubramanniam, P.} 2017. {Approximate controllability of fractional stochastic differential equations driven by mixed fractional Brownian motion via resolvent operators}, Int. J. of Control.90(8), pp. 1713-1727.

\bibitem[Pruss.(1993)] {pruss}
{Pruss, J.,}  1993. {Evolutionary Integral Equations and Applications}. Birkhauser, Basel


\end{thebibliography}
\end{document}